\documentclass[10pt,oneside,reqno]{amsart}

\usepackage{url}		
\usepackage{algorithm2e}
\usepackage{mathtools}
\usepackage{amssymb}
\usepackage{amsthm}

\theoremstyle{plain}
\newtheorem{theorem}{Theorem}[section]

\newtheorem{lemma}[theorem]{Lemma}
\newtheorem{corollary}[theorem]{Corollary}
\newtheorem{definition}{Definition}[section]
\newtheorem{conjecture}{Conjecture}

\theoremstyle{remark}

\DeclarePairedDelimiter \ceil{\lceil}{\rceil}

\begin{document}

\title{ON FINDING THE SMALLEST HAPPY NUMBERS OF ANY HEIGHTS}

\author{GABRIEL LAPOINTE}
\email{gabriel.lapointe18@gmail.com}

\begin{abstract}
This paper focuses on finding the smallest happy number for each height in any numerical base. Using the properties of the height, we deduce a recursive relationship between the smallest happy number and the height where the initial height is function of the numerical base. With the usage of the recursive relationship, we build an algorithm that exploits the properties of the height in order to find all of those smallest happy numbers with unknown height. However, with the modular arithmetic, we conclude on an equation that calculates the smallest happy numbers based on known heights for binary and ternary bases. 
\end{abstract}



\maketitle

\section{Introduction} \label{sec:Introduction}
Researches have been done on smallest happy numbers mainly in decimal base for heights $0$ to $12$ using the modular arithmetic as shown in theorems 2 to 4 \cite{Grundman03} and theorem 3 \cite{Tianxin08}. Moreover, there is an algorithm that searches for the smallest happy number of an unknown height as shown in \cite{Tianxin08}. In this paper, we give and describe this algorithm usable for any numerical base $B \geq 2$. Instead of continuing to search for the smallest happy numbers for any height greater than $12$ in decimal base, we are interested on the binary and ternary bases. Small numerical bases may give hints on an equation or an efficient algorithm to obtain the smallest happy number of any height for larger numerical bases.

Let $x > 0$ be an integer. In the numerical base $B \geq 2$, $x$ can be written as a unique sum 
\begin{equation} \label{eq:DefinitionOfX}
x = \sum_{i = 0}^{L(x)-1} x_iB^i,
\end{equation}
where the positive integers $x_0,\ldots,x_{L(x)-1}$ are the digits of $x$ and $x_j = 0$ for all $j \geq L(x)$. We note $x$ as a vector of its digits in base $B$ as being $x = (x_{L(x)-1}, \ldots, x_0)_B$ where we note $L(x)$ as the length of $x$ and $x(k)$ the $k^{th}$ digit of $x$ starting from the left for $k \in \{0,1,\ldots,L(x)-1\}$.

\begin{definition} \label{dfn:HappyFunction}
Let $B \geq 2$ be the numerical base of a positive integer $x$. We define a function $H_B : \mathbb{N} \longrightarrow \mathbb{N}$ where 
\begin{equation} \label{eq:DefinitionOfH}
\mathcal{H}_B(x) = \sum_{i = 0}^{L(x)-1} x_i^2,
\end{equation}
where $\mathcal{H}_B(0) = 0$ in base $B$.
\end{definition}

To simplify the notation, we note $\mathcal{H}_B^n(x) = \mathcal{H}_B^{n-1}(x) \circ \mathcal{H}_B(x)$ for all $n \geq 1$, where $\circ$ is the function composition operator. The identity is defined as $\mathcal{H}_B^0(x) = x$.

\begin{definition} \label{dfn:HappyNumber}
Let $B \geq 2$ be the numerical base of a positive integer $x$. We call $x$ a happy number if and only if there is $n \in \mathbb{N}$ such that $\mathcal{H}_B^n(x) = 1$.
\end{definition}

\begin{definition} \label{dfn:HappyHeight}
We define the height of a happy number $x$ in base $B \geq 2$ as $\eta_B : \mathbb{N}_* \longrightarrow \mathbb{N}$ where 
\begin{equation}
\eta_B(x) = \min\{\alpha \in \mathbb{N} \; : \; \mathcal{H}^{\alpha}_B(x) = 1\}.
\end{equation}
\end{definition}

\begin{definition} \label{dfn:HappyGamma}
We define the smallest happy number in numerical base $B \geq 2$ for any height by the function $\gamma_B : \mathbb{N} \longrightarrow \mathbb{N}_*$ where 
\begin{equation}
	\gamma_B(\eta(x)) = \min \{x \in \mathbb{N}_* \; : \; \mathcal{H}^{\eta(x)}_B(x) = 1\}.
\end{equation}
\end{definition}

To simplify the notation, we will use $\eta = \eta_B(x)$ and $\gamma(\eta) = \gamma_B(\eta)$. If there is ambiguity, we will precise the numerical base (e.g. $\gamma_3(\eta_3)$).

\section{Smallest happy numbers} \label{sec:Height}
The objective of this section is to show that there is a relationship between subsequent heights of smallest happy numbers for all base $B \geq 2$. Then, we get a general equation for $\gamma$ on which the algorithm will be based.

The following result shows that the function $\gamma$ is defined uniquely for every $\eta$. 

\begin{theorem} \label{thm:ExistUniqueGamma}
For all $\eta \in \mathbb{N}$, there is a unique $x \in \mathbb{N}_*$ such that $\gamma(\eta) = x$ and $\mathcal{H}_B(\gamma(\eta + 1)) \geq \gamma(\eta)$.
\end{theorem}

\begin{proof}
To prove the existence, we have to show that there is a $x \in \mathbb{N}_*$ happy such that $\eta(x) = y$. We proceed by induction on $y$. For $y = 0$, we have $x = 1$ such that $\eta(x) = 0$ for all $B \geq 2$. Lets assume that there is a $x \in \mathbb{N}_*$ happy such that $\eta(x) = y$. Let $z = (1, \ldots, 1)_B$ such that $L(z) = x$. Thus, $\mathcal{H}_B(z) = x$ and $z$ is also happy. In virtue of the induction hypothesis, $\eta(z) = y + 1$ which proves the existence for all $y \in \mathbb{N}$. 

Let $x_1, x_2 \in \mathbb{N}_*$ where $x_1 \neq x_2$. Then, let $x_1 > x_2$ without loss of generality. Having $\gamma(\eta) = x_1$ or $\gamma(\eta) = x_2$ contradict the definition \ref{dfn:HappyGamma}. Thus, $x_1 = x_2$ which proves the uniqueness. 

We deduce from the definition \ref{dfn:HappyGamma} that $\mathcal{H}_B(\gamma(\eta + 1)) \geq \gamma(\eta)$.
\end{proof} 

We note that $1$ is the unique integer of height $0$. If $1 < x < B$, then $\mathcal{H}_B(x) = x^2 > x$. This ensures that $(1,0)_B$ is the smallest happy number of height $1$ for all $B \geq 2$. 

\begin{lemma} \label{lem:GammaUpperBound}
If $\gamma(\eta + 1) \geq B^2$, then $\gamma(\eta + 1) > \gamma(\eta)$.
\end{lemma}

\begin{proof}
By theorem \ref{thm:ExistUniqueGamma}, $\gamma(\eta + 1) \neq \gamma(\eta)$ for all $\eta$. We proceed by contrapositive and suppose that $\gamma(\eta) > \gamma(\eta + 1)$. Since $\gamma(\eta + 1) \geq B^2$, we have, in virtue of the lemma 6 \cite{Grundman01}, that $\mathcal{H}_B(\gamma(\eta + 1)) < \gamma(\eta + 1)$. Applying this with our hypothesis implies that $\mathcal{H}_B(\gamma(\eta + 1)) < \gamma(\eta)$ a contradiction with the theorem \ref{thm:ExistUniqueGamma}. Therefore, $\gamma(\eta+1) > \gamma(\eta)$.
\end{proof}

We want to show that the smallest happy numbers are strictly increasing as their height increases starting at a certain height when $\gamma(\eta + 1) > 2(B-1)^2$. However, even if $B^2 \leq \gamma(\eta + 1) \leq 2(B-1)^2$, it is generally not true. For example, using brute-force, we calculate that $\gamma_{18}(28) = (1, 2, 17)_{18} > B^2$ and $\gamma_{18}(29) = (11, 16)_{18}$. We see that $\gamma_{18}(28) > \gamma_{18}(29)$. The next result ensure that $\{\gamma(\eta), \gamma(\eta + 1), \ldots\}$ is strictly increasing when $\gamma(\eta + 1) > 2(B-1)^2$ for all $B \geq 3$ and $\gamma(\eta + 1) \geq B^2$ for the binary base.

\begin{lemma} \label{lem:Gamma_Increasing_With_Height_Increasing}
If $\gamma(\eta + 1) > 2(B-1)^2$ for all $B \geq 3$ and $\gamma(\eta + 1) \geq B^2$ if $B = 2$, then $\gamma(\eta) < \gamma(\eta + 1) < \ldots < \gamma(\eta + k)$ for all $k \in \mathbb{N}$.
\end{lemma}

\begin{proof}
Let's proceed by induction on $k$. For the base case, we have to show that $\gamma(\eta) < \gamma(\eta + 1)$. Since $\gamma(\eta + 1) > 2(B-1)^2 \geq B^2$ for all $B \geq 3$ and $\gamma(\eta + 1) \geq B^2$ if $B = 2$, we apply the lemma \ref{lem:GammaUpperBound} and then get the result. Suppose that $\gamma(\eta) < \gamma(\eta + 1) < \ldots < \gamma(\eta + k)$ holds for a certain $k$. We have to show that it also holds for $k + 1$. 

Suppose that $\gamma(\eta + k + 1) \leq 2(B-1)^2$ for all $B \geq 3$ and $\gamma(\eta + k + 1) < B^2$ if $B = 2$. In virtue of the corollary 7 \cite{Grundman01}, there is $n \in \mathbb{N}$ such that $\mathcal{H}^n_B(\gamma(\eta + k + 1)) < B^2$. It follows that $\mathcal{H}^{n+1}_B(\gamma(\eta + k + 1)) \leq 2(B-1)^2$ which is maximal if any $\gamma(\eta^{\prime}) = (B-1, B-1)_B$ for all $n < \eta + k + 1$ and $\eta^{\prime} \leq \eta + k + 1$. By induction on $n$ with the theorem \ref{thm:ExistUniqueGamma}, one shows that $\gamma(\eta^{\prime}) \leq 2(B-1)^2$ for all $\eta^{\prime} \leq \eta + k + 1$. Therefore, it is generally false to say that $\gamma(0) < \gamma(1) < \ldots < \gamma(\eta^{\prime})$.

Suppose instead that $\gamma(\eta + k + 1) > 2(B-1)^2$ for all $B \geq 3$ and $\gamma(\eta + k + 1) \geq B^2$ if $B = 2$. By the lemma \ref{lem:GammaUpperBound}, we have $\gamma(\eta + k + 1) > \gamma(\eta + k)$. Using the induction hypothesis, we get $\gamma(\eta) < \gamma(\eta + 1) < \ldots < \gamma(\eta + k) < \gamma(\eta + k + 1)$. Therefore, $\gamma(\eta) < \gamma(\eta + 1) < \ldots < \gamma(\eta + k)$ for all $k \in \mathbb{N}$.
\end{proof}

Let $[x]$ be the integer part function of $x$ and $\lceil x \rceil$ be the ceiling function of $x$. The operator $*$ in the expression $a * b$ denotes the concatenation of $a$ and $b$.

Let $L_{\eta} = \alpha_{\eta} + t_{\eta}$ denotes the total number of digits of $\gamma(\eta)$, where $\alpha_{\eta}$ is the number of digits lower than $B-1$ and $t_{\eta}$ the number of $B-1$. We also note $\gamma(\eta) = A_{\eta} * T_{\eta}$ where $A_{\eta}$ is the integer containing the digits lower than $B-1$ and $T_{\eta}$ the digits $B-1$. The next results give boundaries on $\alpha_{\eta}$ and $t_{\eta}$.

\begin{corollary} \label{crl:HeightLowerTBound}
We have $t_{\eta + 1} \geq \left[\frac{\gamma(\eta)}{(B-1)^2}\right] - (2B - 4)$ for all $B \geq 2$.
\end{corollary}

\begin{proof}
Let $y_{\eta + 1} = \left[\frac{\gamma(\eta)}{(B-1)^2}\right]$. By definition of $t_{\eta + 1}$, we know that 
\begin{equation} \label{eq:NumberOfDigits_LowerThan_BMinusOne}
	\alpha_{\eta + 1} = L_{\eta + 1} - t_{\eta + 1}.
\end{equation}
In virtue of the lemma 2.1 \cite{Mei16}, $y_{\eta + 1} \leq L_{\eta + 1} \leq y_{\eta + 1} + 4$. With the equation \eqref{eq:NumberOfDigits_LowerThan_BMinusOne}, it implies that $\alpha_{\eta + 1} \leq y_{\eta + 1} + 4 - t_{\eta + 1}$ and $y_{\eta + 1} \leq \alpha_{\eta + 1} + t_{\eta + 1}$. With $\gamma(\eta) = (B-1)^2 y_{\eta + 1} = A_{\eta} * T_{\eta}$, we deduce that
\begin{equation} \label{ineq:BoundOnGamma_UsingDigits}
(B-1)^2y_{\eta + 1} \leq (B-2)^2(y_{\eta + 1} + 4 - t_{\eta + 1}) + (B-1)^2t_{\eta + 1}.
\end{equation}
The inequation \eqref{ineq:BoundOnGamma_UsingDigits} is equivalent to
\begin{equation*} 
	(2B - 3)y_{\eta + 1} \leq (4B^2 - 16B + 16) + (2B - 3)t_{\eta + 1}
\end{equation*}
which gives $t_{\eta + 1} \geq y_{\eta + 1} - \frac{4B^2 - 16B + 16}{2B - 3} = y_{\eta + 1} - 2B + 5 - \frac{1}{2B - 3}$. Since $B \geq 2$, we have $0 < \frac{1}{2B - 3} \leq 1$. Therefore, 
\begin{equation*}
t_{\eta + 1} \geq y_{\eta + 1} - (2B - 4).
\end{equation*}
\end{proof}

Since we found the lower bound for $t_{\eta + 1}$, we can deduce an upper bound on $\alpha_{\eta + 1}$.

\begin{corollary} \label{crl:HeightLowerAlphaBound}
We have $\alpha_{\eta + 1} \leq 2B$ for all $B \geq 2$.
\end{corollary}

\begin{proof}
By definition of $\alpha_{\eta + 1}$, we have $\alpha_{\eta + 1} = L_{\eta + 1} - t_{\eta + 1}$. Applying the lemma 2.1 \cite{Mei16} and the corollary \ref{crl:HeightLowerTBound}, we obtain 
\begin{equation*}
\alpha_{\eta + 1} \leq \left[\frac{\gamma(\eta)}{(B-1)^2}\right] + 4 - \left(\left[\frac{\gamma(\eta)}{(B-1)^2}\right] - (2B - 4)\right) = 2B.
\end{equation*}
\end{proof}

Now that we have found boundaries on $\gamma(\eta)$ and its number of digits, we use them in order to demonstrate a recursive relation between $\gamma(\eta + 1)$ and $\gamma(\eta)$ starting with an initial height $\eta = \eta^*$.

\begin{lemma} \label{lem:UpperBound_D}
We have $\mathcal{H}_B(\gamma(\eta + 1)) \leq \gamma(\eta) + 4(B-1)^2$.
\end{lemma}

\begin{proof}
In virtue of the theorem \ref{thm:ExistUniqueGamma}, we have $\mathcal{H}_B(\gamma(\eta + 1)) \geq \gamma(\eta)$ if and only if there is $d \in \mathbb{N}$ such that $\mathcal{H}_B(\gamma(\eta + 1)) = \gamma(\eta) + d$.

By definition of the integer part function, there is a positive integer $v < (B - 1)^2$ such that
\begin{equation}\label{eq:GammaIntegerPartDef}
\gamma(\eta) = (B-1)^2 \left[\frac{\gamma(\eta)}{(B-1)^2}\right] + v.
\end{equation}
In virtue of the lemma 2.1 \cite{Mei16} and the equation \eqref{eq:GammaIntegerPartDef}, we have 
\begin{equation} \label{ineq:UpperBound_GammaIntegerPartDef}
	\mathcal{H}_B(\gamma(\eta + 1)) = \gamma(\eta) + d \leq (B-1)^2 \left(\left[\frac{\gamma(\eta)}{(B-1)^2}\right] + 4 \right) + v.
\end{equation}
Equivalently, with the equation \eqref{eq:GammaIntegerPartDef}, the inequation \eqref{ineq:UpperBound_GammaIntegerPartDef} can be rewritten as 
\begin{equation} \label{ineq:NotSimplifiedUpperBound_GammaIntegerPartDef}
	\gamma(\eta) + d \leq \gamma(\eta) + 4(B-1)^2.
\end{equation}
After simplification of the inequation \eqref{ineq:NotSimplifiedUpperBound_GammaIntegerPartDef} we get
\begin{equation*}
d \leq 4(B-1)^2.
\end{equation*}
\end{proof}

\begin{theorem} \label{thm:Recursive_Equation_With_Initial_Condition}
Let $\eta^*$ be the smallest height such that $\gamma(\eta^*) > 2(B-1)^3 + 5$. If $\gamma(\eta^*) > 2(B-1)^3 + 5$, then $\mathcal{H}_B(\gamma(\eta + 1)) = \gamma(\eta)$ for all $B \geq 2$ and $\eta \geq \eta^*$.
\end{theorem}

\begin{proof}
In virtue of the theorem \ref{thm:ExistUniqueGamma}, there is $d \in \mathbb{N}$ such that $\mathcal{H}_B(\gamma(\eta + 2)) \geq \gamma(\eta + 1) + d$. The lemma \ref{lem:UpperBound_D} implies that $d \leq 4(B-1)^2$. Let's proceed by induction on $\eta$ and suppose that $\eta = \eta^*$ and $d > 0$. We want to show that $d > 0$ is impossible in order to show that $d = 0$.

We know that $\gamma(\eta + 1) = A_{\eta + 1} * T_{\eta + 1}$. If $t_{\eta + 1} = 0$, we have $\gamma(\eta + 1) + d = A_{\eta + 1} + d$. This gives either $\gamma(\eta + 1) + d = A^{\prime}_{\eta + 1}$, where $A^{\prime}_{\eta + 1}$ contains only digits between $1$ and $B-2$, or $\gamma(\eta + 1) + d = A^{\prime}_{\eta + 1} * T^{\prime}_{\eta + 1}$ where $T^{\prime}$ contains only the digits $B-1$. The second case will be taken as if $t_{\eta + 1} > 0$. If $\gamma(\eta + 1) + d = A^{\prime}_{\eta + 1}$, by the corollary \ref{crl:HeightLowerAlphaBound}, we have that 
\begin{equation}
	\mathcal{H}_B(\gamma(\eta + 1) + d) \leq 2B(B-2)^2.
\end{equation} 

Suppose now that $t_{\eta + 1} > 0$. This means that $\gamma(\eta + 1) = A_{\eta + 1} * (B-1) * \ldots * (B-1)$. Since $d > 0$, we have that $\gamma(\eta + 1) + d = (A_{\eta + 1} + 1) * 0 * \ldots * 0 * (d - 1)$. Thus, we have that
\begin{equation} \label{eq:H_Of_Gamma_Eta_Plus_D}
	\mathcal{H}_B(\gamma(\eta + 1) + d) = \mathcal{H}_B(A_{\eta + 1} + 1) + \mathcal{H}_B(d - 1).
\end{equation}
It follows from the equation \eqref{eq:H_Of_Gamma_Eta_Plus_D} and the corollary \ref{crl:HeightLowerAlphaBound} that
\begin{equation} \label{ineq:H_Of_A_Eta_Plus_One_UpperBound}
	\mathcal{H}_B(A_{\eta + 1} + 1) \leq (2B - 1)(B-2)^2 + (B-1)^2. 
\end{equation}
We know that $d-1 \leq 4(B-1)^2 - 1 = 4B^2 - 8B + 3 = (3, B-8, 3)_{B \geq 8}$. We deduce that
\begin{equation} \label{ineq:H_Of_D_Minus_One_UpperBound}
	\mathcal{H}_B(d-1) \leq 2(B-1)^2 + 4
\end{equation}
because to have $\mathcal{H}_B(d-1)$ maximal, we have to have $d - 1 = (2, B-1, B-1)_B$. Combining the inequations \eqref{ineq:H_Of_A_Eta_Plus_One_UpperBound} and \eqref{ineq:H_Of_D_Minus_One_UpperBound} in the equation \eqref{eq:H_Of_Gamma_Eta_Plus_D} gives
\begin{equation} \label{ineq:H_Of_Gamma_Eta_Plus_D_UpperBound}
	\begin{aligned}
		\mathcal{H}_B(\gamma(\eta + 1) + d) &\leq (2B-1)(B-2)^2 + (B-1)^2 + 2(B-1)^2 + 4 \\
		&= 2(B-1)^3 + 5.
	\end{aligned}
\end{equation}
However, $\gamma(\eta) > 2(B-1)^3 + 5$ which contradict the theorem \ref{thm:ExistUniqueGamma} because $\mathcal{H}^2_B(\gamma(\eta + 2)) = \mathcal{H}_B(\gamma(\eta + 1) + d) < \gamma(\eta)$. Therefore, we have to have $d = 0$.

Suppose now that $\mathcal{H}_B(\gamma(\eta + 1)) = \gamma(\eta)$ for a certain $\eta > \eta^*$. We need to show that it holds for $\eta + 1$. In virtue of the theorem \ref{thm:ExistUniqueGamma}, there is $d \in \mathbb{N}$ such that $\mathcal{H}_B(\gamma(\eta + 2)) = \gamma(\eta + 1) + d$. Suppose that $d > 0$. Applying $\mathcal{H}_B$ on both sides gives
\begin{equation}
	\mathcal{H}^2_B(\gamma(\eta + 2)) = \mathcal{H}_B(\gamma(\eta + 1) + d).
\end{equation}
Since $\gamma(\eta^*) > 2(B-1)^3 + 5 > \max(2(B-1)^2, B^2)$ for all $B \geq 2$, the lemma \ref{lem:Gamma_Increasing_With_Height_Increasing} implies that $\gamma(\eta) > \gamma(\eta^*)$. Therefore, we get the same contradiction as in the method used in the induction base case. Thus, $d = 0$ and then, using the induction hypothesis, we obtain $\mathcal{H}^2_B(\gamma(\eta + 2)) = \mathcal{H}_B(\gamma(\eta + 1)) = \gamma(\eta)$.

Therefore, $\mathcal{H}_B(\gamma(\eta + 1)) = \gamma(\eta)$ for all $B \geq 2$ and $\eta \geq \eta^*$ if $\gamma(\eta^*) > 2(B-1)^3 + 5$.
\end{proof}

With the result given by the theorem \ref{thm:Recursive_Equation_With_Initial_Condition}, we deduce the recursive relation between $\gamma(\eta + 1)$ and $\gamma(\eta)$ with the following corollary.

\begin{corollary} \label{crl:GammaCanonicalForm}
If $\eta \geq \eta^*$, then
\begin{equation} \label{eq:General_Gamma_Height}
	\gamma(\eta + 1) = (A_{\eta + 1}+1) B^{\frac{\gamma(\eta) - \mathcal{H}_B(A_{\eta + 1})}{(B-1)^2}} - 1 
\end{equation}
\end{corollary}

\begin{proof}
By definition of $\gamma$, we can write
\begin{equation} \label{eq:GammaEtaPlusRPlusOne}
\gamma(\eta + 1) = A_{\eta + 1} * T_{\eta + 1} = (A_{\eta + 1}+1) B^{t_{\eta + 1}} - 1.
\end{equation}
Applying $\mathcal{H}_B$ to the equation \eqref{eq:GammaEtaPlusRPlusOne} gives
\begin{equation}
\mathcal{H}_B(\gamma(\eta + 1)) = \mathcal{H}_B(A_{\eta + 1}) + (B-1)^2 t_{\eta + 1}.
\end{equation}
Since $\eta \geq \eta^*$, we apply the theorem \ref{thm:Recursive_Equation_With_Initial_Condition} and get 
\begin{equation}
\gamma(\eta) = \mathcal{H}_B(A_{\eta + 1}) + (B-1)^2 t_{\eta + 1}
\end{equation}
if and only if
\begin{equation} \label{eq:NumberOfBMinusOne}
t_{\eta + 1} = \frac{\gamma(\eta) - \mathcal{H}_B(A_{\eta + 1})}{(B-1)^2}.
\end{equation}
Substituting the equation \eqref{eq:NumberOfBMinusOne} in \eqref{eq:GammaEtaPlusRPlusOne} gives the result. 
\end{proof}

We conclude this section on the following questions: 
\begin{enumerate}
	\item How can the initial height $\eta^*$ be expressed in function of the base $B$?
	\item Can we get a sharper lower bound on $\gamma(\eta^*)$?
\end{enumerate}

\section{Algorithm evaluating $\gamma(\eta)$ with $\eta$ unknown} \label{sec:HappyAlgorithm}
In this section, the objective is to use the corollary \ref{crl:GammaCanonicalForm} in order to build two algorithms. The first one searches for the minimal $A_{\eta + 1}$ for an unknown height $\eta$ and the second one searches for $\gamma(\eta + 1)$ corresponding to $A_{\eta + 1}$ found. We will see that evaluating the remainder of $\frac{\gamma(\eta)}{(B-1)^2}$ is an important computational obstacle to consider. Also, we will explain why these algorithms cannot search for $\gamma(\eta)$ based on a given $\eta$.

For this section, we define a function $U: \mathbb{N} \setminus \{0,1\} \longrightarrow \mathbb{N}$ taking the base $B \geq 2$ as the input parameter. The function $U(B)$ denotes the upper bound of iterations done in the algorithm \ref{alg:GammaEta}. In virtue of lemma 2.1 \cite{Mei16}, $U(B)$ exists and is finite.

Let $E$ be the set of solutions of $\mathcal{H}_B(A_{\eta + 1})$. The objective of the algorithm \ref{alg:AMinimal} is to find $\hat{A}_{\eta + 1} = \min(E)$. 

\newpage

\RestyleAlgo{boxed}
\begin{algorithm}[H] \label{alg:AMinimal}
\SetAlgoLined
\KwData{$\mathcal{H}_B(A_{\eta + 1}) \in \mathbb{N}$ and $B \geq 3$}
\KwResult{$\hat{A}_{\eta + 1}$}
\BlankLine
\lIf{$\mathcal{H}_B(A_{\eta + 1}) = 0$}{\Return $0$}
$D \leftarrow \left[\frac{\mathcal{H}_B(A_{\eta + 1})}{(B-2)^2}\right] + 1$ \\
\If{$\mathcal{H}_B(A_{\eta + 1}) \equiv 0 \pmod{(B-2)^2}$}
{
	\Return $(B-2, \ldots, B-2)_B$ with $D - 1$ times the digit $B-2$
}
$A^* \leftarrow (1, \ldots, 1)_B$ with $D$ times the digit $1$ \\
$j \leftarrow D-1$ \\
$h \leftarrow \mathcal{H}_B(A^*)$ \\
\While{$h \neq \mathcal{H}_B(A_{\eta + 1})$}
{
	\If{$h > \mathcal{H}_B(A_{\eta + 1})$ \textbf{or} $A^*(j) = B - 2$}
	{
		\If{$j = 0$}
		{
			$D \leftarrow D + 1$ \\
			$A^* = (1, \ldots, 1)_B$ with $D$ times the digit $1$ \\
		}
		\Else
		{
			$j \leftarrow j - 1$ \\
			$A^*(j) \leftarrow A^*(j) + 1$ \\
			$A^*(k) \leftarrow A^*(j)$ for all $k = j+1, \ldots, D-1$ \\
		}
	}
	\Else
	{
		$j \leftarrow D-1$ \\
		$A^*(j) \leftarrow A^*(j) + 1$ \\
	}
	$h \leftarrow \mathcal{H}_B(A^*)$ \\
}
\Return{$A^*$}
\caption{Find $\hat{A}_{\eta + 1} = \min(E)$}
\end{algorithm}

\begin{theorem} \label{thm:AMinimalAlgorithm}
For any $B \geq 3$, the algorithm \ref{alg:AMinimal} terminates with any $\mathcal{H}_B(A_{\eta + 1}) \geq 0$ and returns $\hat{A}_{\eta + 1}$ for an unknown $\eta$.
\end{theorem}

\begin{proof}
If $\mathcal{H}_B(A_{\eta + 1}) = 0$, then $\hat{A}_{\eta + 1} = 0$ and the algorithm terminates and returns $0$. If $0 < \mathcal{H}_B(A_{\eta + 1}) \equiv 0 \pmod{(B-2)^2}$ and $B \geq 3$, then we get $\hat{A}_{\eta + 1} = (B-2, \ldots, B-2)_B$ with $\frac{\mathcal{H}_B(A_{\eta + 1})}{(B-2)^2}$ times the $(B-2)$-digit. This solution is trivially minimal because the other solutions, if they exist, have more than $\frac{\mathcal{H}_B(A_{\eta + 1})}{(B-2)^2}$ digits.

Suppose that $\mathcal{H}_B(A_{\eta + 1}) \not \equiv 0 \pmod{(B-2)^2}$ and $B \geq 3$. We have to show that the algorithm stops when the current candidate $A^*$ is the minimal solution $\hat{A}_{\eta + 1}$. Let $\bar{A} = (1, \ldots, 1)_B$ with $\mathcal{H}_B(A_{\eta + 1})$ times the digit 1. Trivially, $\bar{A} \geq \hat{A}_{\eta + 1}$ where $\bar{A} = \max(E)$.

By construction, the algorithm initializes $D = \left[\frac{\mathcal{H}_B(A_{\eta + 1})}{(B-2)^2}\right] + 1$ and $A^* = (1, \ldots, 1)_B$ with $D$ times the digit $1$ and runs through a subset $N \subset \mathbb{N}_*$ of numbers. Let $j = D - 1$ be an iterator where $A^* = (A^*(0), A^*(1), \ldots, A^*(j), \ldots, A^*(D-1))_B$ and $h = \mathcal{H}_B(A^*)$.

Let $i$ be a counter of iterations in the loop and $A^*_i$ be the value of $A^*$ at the $i^{th}$ iteration. For all $i$, the algorithm ensures that $A^*_{i+1} > A^*_i$ because in every case, it either increments a digit of $A^*_i$ or increments $D$ and fix $A^*_{i+1} = (1, \ldots, 1)_B$ with $D$ times the digit $1$.

Let $I = \{1, 2, \ldots, B-2\}$ and suppose that $A^*_{i-1}(r) \in I$ where $0 \leq r < D$ at each iteration $i$. We have to show that $A^*_i(r) \in I$. We can observe three general cases:
\begin{enumerate}
	\item If $A^*_i(j) < B-2$, then $A^*_i(j)$ increases by one implying that $A^*_i(j) \leq B-2 \in I$.
	\item If $j = 0$, either the algorithm falls into the case (i) or $A^*_i(r) = 1 \in I$ for all $r$.
	\item If $j > 0$, either the algorithm falls into the case (i) or $A^*_i(k) = A^*_i(j) + 1 \in I$ for all $j \leq k \leq D-1$ where $j = j - 1$ because by definition of $j$, we know that $A^*_i(j-1) \in I$.
\end{enumerate}
Therefore, $A^*_i(r) \in I$ for all $0 \leq r < D$ at each iteration $i$.

If $j > 0$ and $A^*_i(j) = B - 2$, then $A^*_i(k) = A^*_i(j)$ for all $k = j+1, \ldots, D-1$. The values when $1 \leq A^*_i(k) \leq A^*_i(j)$ are ignored because the addition is commutative in $\mathbb{N}_*$ and we run through $N$ in ascending order. Also, $\mathcal{H}_B$ is invariant to the permutations of the digits of $A^*_i$. Thus, the algorithm has already tested those values where $1 \leq A^*_i(k) \leq A^*_i(j) \leq B-2$ up to permutation. Therefore, if $A^*_i$ is an ignored value, then $A^*_i \neq \hat{A}_{\eta + 1}$. This also means that 
\begin{equation} \label{ineq:A_By_Digit}
	1 \leq A^*_i(D-1) \leq A^*_i(D-2) \leq \ldots \leq A^*_i(0) \leq B-2.
\end{equation}

We have to show that the algorithm ends with $A^* \in E$. We have shown that $A^*_i(r) \in I$ for all $r$ and that each time $D$ increments, $A^*$ is reset to $(1,\ldots,1)_B$ with $D$ times the digit one. We have also shown that $A_{i+1} > A_i$ for all $i$. Thus, the loop either stops when $h = \mathcal{H}_B(A_{\eta + 1})$ or reaches its upper bound when $A^*_i = \bar{A}$. Therefore, the algorithm ends with $A^* \in E$.

Let $K = \left\{\left[\frac{\mathcal{H}_B(A_{\eta + 1})}{(B-2)^2}\right] + 1, \left[\frac{\mathcal{H}_B(A_{\eta + 1})}{(B-2)^2}\right] + 2, \ldots, \mathcal{H}_B(A_{\eta + 1})\right\}$. We have to show that $D \in K$. If $D > \mathcal{H}_B(A_{\eta + 1})$, then $A^* > \bar{A}$ a contradiction with $A^* \in E$. If $D < \left[\frac{\mathcal{H}_B(A_{\eta + 1})}{(B-2)^2}\right] + 1$, then $\mathcal{H}_B(A^*_i) \leq (D - 1)(B-2)^2 < \mathcal{H}_B(A_{\eta + 1})$ which is also a contradiction with $A^* \in E$. Thus, $D \in K$. 

Since $A^*$ increases, $A^* \in E$ and the ignored values cannot be the minimal solution searched, the algorithm returns $\hat{A}_{\eta + 1}$ for any $\mathcal{H}_B(A_{\eta + 1}) \geq 0$ and $B \geq 3$.
\end{proof}

In order to find $\gamma(\eta + 1)$, we need to know $t_{\eta + 1}$. From the corollary \ref{crl:GammaCanonicalForm}, we deduce that there is $m \in \mathbb{N}$ such that $\gamma(\eta) - \mathcal{H}_B(A_{\eta + 1}) = m(B - 1)^2$. Equivalently,
\begin{equation}\label{eq:HRest}
\gamma(\eta) \equiv \mathcal{H}_B(A_{\eta + 1}) \pmod{(B - 1)^2},
\end{equation}
where $\eta \geq \eta^*$.

The second objective of this section is to build an algorithm that searches for $\gamma(\eta + 1)$. We note $\gamma^j$ the $j^{th}$ possible candidate of $\gamma(\eta + 1)$.

\RestyleAlgo{boxed}
\begin{algorithm}[H]\label{alg:GammaEta}
\SetAlgoLined
\KwData{$R_0 = \mathcal{H}_B(A_{\eta + 1}) \in \mathbb{N}$ and $B \geq 3$}
\KwResult{$\gamma = (A^*_{j^*}, R_{j^*})$, where $j^* = \arg \min_j \gamma^j(\eta + 1)$}
\BlankLine
\lIf{$R_0 = a^2$ where $0 \leq a \leq B-2$}{\Return{$(\sqrt{R_0}, R_0)$}} 
$A_0^* \leftarrow A \leftarrow $ Algorithm\ref{alg:AMinimal}$(R_0, B)$ \\
$\gamma^0 \leftarrow (A_0^*, R_0)$ \\
$j \leftarrow 1$ \\
\While{$j \leq \left[\frac{3B}{2}\right] - 3$}
{
	$R_j \leftarrow R_0 + j(B-1)^2$ \\
	$A_j^* \leftarrow $ Algorithm\ref{alg:AMinimal}$(R_j, B)$ \\
	\If{$\left[\frac{A_j^*}{B^j}\right] < A$}
	{
		$\gamma^j \leftarrow (A_j^*, R_j)$ \\
		$A \leftarrow \left[\frac{A_j^*}{B^j}\right]$
	}
	$j \leftarrow j + 1$ \\
}
\Return{$\min_{j}\gamma^j(\eta + 1)$}
\caption{Find $R_{j^*}$ and $A^*_{j^*}$ such that $\gamma(\eta + 1) = \min_{j}\gamma^j(\eta + 1) = (A^*_{j^*}, R_{j^*})$ for an unknown $\eta$}
\end{algorithm}

\begin{theorem} \label{thm:GammaEtaAlgorithm}
The algorithm \ref{alg:GammaEta} ends and returns $\gamma(\eta + 1) = \min_{j}\gamma_j(\eta + 1)$ for an unknown $\eta$ and for any $\mathcal{H}_B(A_{\eta + 1}) = R_0 \geq 0$ where $B \geq 3$, $U_B = \left[\frac{3B}{2}\right] - 3$ and $j \in \{0, 1, 2, \ldots, U_B\}$.
\end{theorem}

\begin{proof}
Let $j \in \mathbb{N}$ be the iterations in the loop and $\eta$ an unknown height. In virtue of the theorem \ref{thm:ExistUniqueGamma}, we know that for all height $\eta$, there is a unique positive integer $\hat{\gamma}$ such that $\gamma(\eta + 1) = \hat{\gamma}$. Thus, we have to prove that $\hat{\gamma} \in \{\gamma^j(\eta + 1)\}_{j = 0}^{U_B}$. By corollary \ref{crl:HeightLowerAlphaBound}, we have $U_B = 2B$. However, we show that $U_B$ can be improved in order to reduce the number of iterations processed by the algorithm. It will then follow that $\hat{\gamma} \in \{\gamma^j(\eta + 1)\}_{j = 0}^{U_B}$.

By construction, the algorithm initializes $\gamma^0(\eta + 1) = (A_0^*, R_0)$, where $A = A_0^* = \min(E)$ per algorithm \ref{alg:AMinimal}. If we remove $j$ times the digit $B-1$, then using the corollary \ref{crl:GammaCanonicalForm} gives
\begin{align}
t_{\eta + 1} - j &= \frac{\gamma(\eta) - R_0}{(B-1)^2} - j \notag \\
				 &= \frac{\gamma(\eta) - (R_0 + j(B-1)^2)}{(B-1)^2}.
\end{align}
Let $R_j = R_0 + j(B-1)^2$. We apply the algorithm \ref{alg:AMinimal} on $R_j$ to output $A_j^*$ at iteration $j$ where $R_j = \mathcal{H}_B(A_j^*)$. Since we removed $j$ times the digit $B-1$, they have to be added back to $A_j^*$ by respecting the constraints $1 \leq a^*_i \leq B-2$. Lagrange's theorem states that every positive integer can be written as a sum of 4 squares. This implies that $1 \leq L(A_0^*) \leq 4$ where $L(A_0^*) = 1$ if and only if $R_0$ is a perfect square. In such case, $\hat{\gamma} = \gamma^0(\eta + 1) = (\sqrt{R_0}, R_0)$. 

Suppose that $R_0$ is not a perfect square. It follows that $2 \leq L(A_0^*) \leq 4$. Because $(B - 1)^2 > (B - 2)^2$, then the lower bound of $L(A_j^*)$ increases of $1$ when
\begin{equation}
2 + \frac{j(B-1)^2}{(B-2)^2} = 2 + \frac{j(2B-3)}{(B-2)^2} = 3
\end{equation}
or equivalently if 
\begin{equation}
j\frac{2B - 3}{(B-2)^2} = 1.
\end{equation}
Hence, once
\begin{equation} 
	j = \ceil*{\frac{(B-2)^2}{2B-3}} = \frac{B}{2} - 1.
\end{equation}
because $j \in \mathbb{N}$. Thus, $L(A_j^*) \geq 5$ once $j = \ceil*{\frac{3B}{2}} - 3$ which ensures that $\hat{\gamma} \not \in \{\gamma^j(\eta + 1)\}_{k = \ceil*{\frac{3B}{2}} - 3}^{\infty}$. Therefore, we deduce that 
\begin{equation} \label{ineq:MaxIterationUpperBound}
	U_B = \left[\frac{3B}{2}\right] - 3
\end{equation}
and $\hat{\gamma} \in \{\gamma^j(\eta + 1)\}_{j = 0}^{U_B}$.

We have to show that the algorithm returns $\hat{\gamma}$. If $\left[\frac{A_j^*}{B^j}\right] < A$, then $\gamma^j(\eta + 1) < \gamma^l(\eta + 1)$, where $0 \leq l < j$ and $\left[\frac{A_j^*}{B^j}\right]$ means we remove the last $j$ digits to $A_j$. Thus, we only keep the minimal solution $\gamma^j(\eta + 1)$, for iterations $l \leq j$, and we update $A = \left[\frac{A_j^*}{B^j}\right]$ to ensure that $A$ is always minimal. Therefore, there is $j$ such that the algorithm returns $\hat{\gamma} = \min_j \{\gamma^j(\eta + 1)\}_{j = 0}^{U_B}$ for all $B \geq 3$.
\end{proof}

The algorithms \ref{alg:AMinimal} and \ref{alg:GammaEta} have been implemented in the C++ language. The source code is available at \url{https://github.com/glapointe7/SmallestHappyNumbers}. To test the algorithm \ref{alg:GammaEta}, we validated our results with those of \cite{Tianxin08}(page 1924) in the decimal base.

In the table \ref{tbl:Comparison_UB_With_J}, we compared our upper bound $U_B$ to the number of iterations $j$ maximal, noted $J$, among integers $0 \leq R < (B-1)^2$ in order to obtain $\hat{\gamma}$ for bases $3 \leq B \leq 24$.
\newpage
\begin{table}[htpb]
\caption{Comparison of $U_B$ with $J$ for $3 \leq B \leq 24$.}
{\begin{tabular}{|l|c|r||l|c|r|}
  \hline
  $B$ & $J$ & $U_B$ & $B$ & $J$ & $U_B$ \\
  \hline
  3 & 0 & 1 & 14 & 6 & 18 \\ \hline
  4 & 1 & 3 & 15 & 6 & 19 \\ \hline
  5 & 1 & 4 & 16 & 6 & 21 \\ \hline
  6 & 2 & 6 & 17 & 7 & 22 \\ \hline
  7 & 2 & 7 & 18 & 7 & 24 \\ \hline
  8 & 3 & 9 & 19 & 7 & 25 \\ \hline
  9 & 3 & 10 & 20 & 7 & 27 \\ \hline
  10 & 5 & 12 & 21 & 8 & 28 \\ \hline
  11 & 6 & 13 & 22 & 10 & 30 \\ \hline
  12 & 4 & 15 & 23 & 10 & 31 \\ \hline
  13 & 5 & 16 & 24 & 8 & 33 \\ \hline
\end{tabular}}
\label{tbl:Comparison_UB_With_J}
\end{table}

We conclude this section on the following questions:
\begin{enumerate}
	\item Can we improve the time complexity of the algorithms \ref{alg:AMinimal} and \ref{alg:GammaEta}?
	\item According to the table \ref{tbl:Comparison_UB_With_J}, our upper bound $U_B$ could be sharper. How can we get a sharper upper bound $U_B$?
\end{enumerate}

\section{Smallest Happy Numbers in Binary Base} \label{sec:BinaryBase}
In this section, the objective is to obtain an equation that calculates $\gamma_2(\eta)$ for all $\eta$ using the corollary \ref{crl:GammaCanonicalForm}. After computing $\gamma_2(\eta)$ for $\eta \in \{0,1,2,3,4\}$ we get $\gamma_2(0) = (1)_2$, $\gamma_2(1) = (1,0)_2$, $\gamma_2(2) = (1,1)_2$, $\gamma_2(3) = (1,1,1)_2$ and $\gamma_2(4) = (1,1,1,1,1,1,1)_2$. We see that $\mathcal{H}_2(\gamma_2(\eta + 1)) = \gamma_2(\eta)$ for $\eta \in \{0, 1, 2, 3\}$. Moreover, the smallest height when $\gamma_2(\eta + 1) > 2(B-1)^3 + 5 = 7 = (1, 1, 1)_2$ is $\eta^* = 4$ per theorem \ref{thm:Recursive_Equation_With_Initial_Condition}. Thus, we showed that $\mathcal{H}_2(\gamma_2(\eta + 1)) = \gamma_2(\eta)$ for all $\eta \in \mathbb{N}$ in binary base. Therefore, we choose the initial height $\eta^* = 1$ to fit with the corollary \ref{crl:GammaEquationBinaryBase}.

\begin{corollary} \label{crl:GammaEquationBinaryBase}
For all $\eta \geq 1$, we have
\begin{equation} \label{eq:GammaBinary}
\gamma_2(\eta + 1) = 2^{\gamma_2(\eta)} - 1
\end{equation}
and its inverse
\begin{equation} \label{eq:GammaBinaryInverse}
\gamma_2(\eta) = \log_2(\gamma_2(\eta + 1) + 1).
\end{equation}
\end{corollary}

\begin{proof}
For $\eta = 1,2$ the result is obvious. For $\eta \geq \eta^*$, we can apply the corollary \ref{crl:GammaCanonicalForm} where $A_{\eta + 1} = 0$ because $B = 2$. Therefore, the equation \eqref{eq:GammaBinary} is directly obtained and then, we deduce the equation \eqref{eq:GammaBinaryInverse}.
\end{proof}

Using the tetration, which is the iterated exponentiation, the equation \eqref{eq:GammaBinary} can be defined analytically. We present a new notation inspired from \cite{Knuth76} to simplify the notations.

\begin{definition} \label{dfn:KnuthAdapted}
Let $k, x, y, z \in \mathbb{R}$ and $n \in \mathbb{N}$. We define the adapted Knuth's up-arrow notation as being
\begin{equation}
kxy^{xy^{\cdot^{\cdot^{xy^z}}}} = k(xy \uparrow\uparrow n)^z,
\end{equation}
where $k$ is a scalar and $xy$ is repeated $n$ times. The real multiplication operator is used between $k,x$ and $y$. If $n = 0$, then $k(xy \uparrow\uparrow n)^z = k$.
\end{definition}

Knowing that the initial condition is $\gamma_2(\eta^*) = 2 = (1, 0)_2$, we have 
\begin{equation}
	\gamma_2(2) = 2^{\gamma_2(1)} - 1 = 2^2 - 1 = 2^{3 - 1} - 1 = 2^{-1}2^3 - 1 = (2^{-1}2 \uparrow\uparrow 1)^3 - 1.
\end{equation}
If we continue with one more iteration, we get
\begin{equation}
	\gamma_2(3) = 2^{\gamma_2(2)} - 1 = 2^{2^{-1}2^3 - 1} - 1 = 2^{-1}2^{2^{-1}2^3} - 1 = (2^{-1}2 \uparrow\uparrow 2)^3 - 1.
\end{equation}
Using the induction on $\eta > \eta^*$, we obtain the equation
\begin{equation} \label{eq:GammaBinarySolved}
\gamma_2(\eta) = 2^{-1}2^{2^{-1}2^{\cdot^{\cdot^{2^{-1}2^3}}}} - 1 = (2^{-1}2 \uparrow\uparrow (\eta - \eta^*))^3 - 1.
\end{equation}

Recall that a Mersenne's prime number $m$ is a prime number $m = 2^p - 1$, where $p$ is prime. In virtue of the corollary \ref{crl:GammaEquationBinaryBase}, we note that $\{\gamma_2(\eta + 1) : \eta \geq 1\}$ is a possible subset of the Mersenne's prime number set. In particular, $\gamma_2(\eta)$ is a Mersenne's prime number for $\eta \in \{2,3,4,5\}$ which are given by \cite{Sloane64}. This leads us to the following conjecture because nowadays, we cannot tell if $\gamma_2(6) = 2^{2^{127} - 1} - 1$ is a Mersenne's prime number or not. This holds also for all $\eta > 6$.

\begin{conjecture}
For all $\eta \geq 2$, $\gamma_2(\eta)$ is a Mersenne's prime number.
\end{conjecture}

\section{Smallest Happy Numbers in Ternary Base} \label{sec:GammaBase3}
In this section, the objective is to find an equation that finds $\gamma_3(\eta)$ for all $\eta > \eta^*$ by using the relation \eqref{eq:HRest}. We have to calculate the remainder of $\frac{\gamma_3(\eta)}{4}$ and find $\eta > \eta^*$ such that 
\begin{equation}
	\gamma_3(\eta + 1) = (A_{\eta + 1}+1) \cdot 3^{\frac{\gamma_3(\eta) - \mathcal{H}_3(A_{\eta + 1})}{4}} - 1.
\end{equation}
By brute-force, we obtain that $\gamma_3(2) = (1,1,1)_3$, $\gamma_3(3) = 2 \cdot 3^3 - 1$ and  $\gamma_3(4) = 2 \cdot 3^{13} - 1$. In virtue of the theorem \ref{thm:Recursive_Equation_With_Initial_Condition}, the smallest height when $\gamma_3(\eta) > 2(B-1)^3 + 5 = 21 = (2, 1, 0)_3$ is $\eta^* = 3$. We showed that $\mathcal{H}_3(\gamma_3(\eta + 1)) = \gamma_3(\eta)$ for all $\eta \in \mathbb{N}$. Then, we choose $\eta^* = 2$ to fit with the lemma \ref{lem:GammaBase3General}.

By definition of $A_{\eta+1}$, the only digit that can be contained in $A_{\eta+1}$ is one in ternary base. After applying the algorithm \ref{alg:GammaEta} for $\mathcal{H}_3(A_{\eta + 1}) < (B-1)^2 = 4$, we found that $A_{\eta+1} \leq (1,1,1)_3$. Let $\mathcal{A}_3$ denotes the set of all possible values of $A_{\eta+1}$ in ternary base where
\begin{equation}\label{eq:A3SetNotReduced}
\mathcal{A}_3 = \{(0)_3, (1)_3, (1,1)_3, (1,1,1)_3\}.
\end{equation}

In order to find $\gamma_3(5)$, we calculate $\mathcal{H}_3(A_5)$ by using the relation \eqref{eq:HRest}. We find that the rest of the division of $\gamma_3(4)$ by $4$ is $3188645 \equiv 1 \pmod{4}$. Thus, $\mathcal{H}_B(A_5) = 1$ and then, $A_5+1 = 2$. Therefore, 
\begin{equation}\label{eq:Gamma5Base3}
\gamma_3(5) = 2 \cdot 3^{\frac{\gamma_3(4) - 1}{4}} - 1. 
\end{equation}

Note that $\gamma_3(\eta) = 2 \cdot 3^t - 1$ for $\eta = 3,4,5$. Using induction on $t \in \mathbb{N}$, we get that $3^t \equiv 1$ or $3 \pmod{4}$. Thus, 
\begin{equation}\label{eq:GammaBase3General}
2 \cdot 3^t - 1 \equiv 1 \pmod{4}
\end{equation}
for all $t \in \mathbb{N}$.

We see that $\gamma_3(\eta) \equiv 1 \pmod{4}$ for $\eta = 3,4,5$. Let's generalize the relation \eqref{eq:GammaBase3General} for all $\eta > \eta^*$.

\begin{lemma}\label{lem:GammaBase3General}
For all $\eta \geq \eta^*$ in ternary base, we have
\begin{equation}\label{eq:GammaBase3All}
\gamma_3(\eta + 1) = 2 \cdot 3^{\frac{\gamma_3(\eta) - 1}{4}} - 1.
\end{equation}
\end{lemma}

\begin{proof}
We proceed by induction on $\eta$. For $\eta = \eta^*$, we have $\gamma(3) = 2 \cdot 3^{3} - 1$ where $3 = \frac{\gamma_3(2) - 1}{4}$ with $\gamma(2) = 13 = (1, 1, 1)_3$.
Suppose that $\gamma_3(\eta + 1) = 2 \cdot 3^{\frac{\gamma_3(\eta) - 1}{4}} - 1$ for a certain $\eta > \eta^*$. We have $\gamma_3(\eta + 2) = (A_{\eta + 2}+1) \cdot 3^{\frac{\gamma_3(\eta + 1) - \mathcal{H}_3(A_{\eta + 2})}{4}} - 1$. Using the induction hypothesis and knowing that $\frac{\gamma_3(\eta) - 1}{4} \in \mathbb{N}$, we deduce that $\gamma_3(\eta + 1)$ is written in the same form as described by the relation \eqref{eq:GammaBase3General} where $t = \frac{\gamma_3(\eta) - 1}{4}$. Thus, $\gamma_3(\eta + 1) \equiv 1 \pmod{4}$ and then $\mathcal{H}_3(A_{\eta + 2}) \equiv 1 \pmod{4}$. In virtue of the set of possible integers in $A_{\eta + 1}$ \eqref{eq:A3SetNotReduced}, we get $(A_{\eta + 2}+1) = 2$. Therefore, $\gamma_3(\eta + 2) = 2 \cdot 3^{\frac{\gamma_3(\eta + 1) - 1}{4}} - 1$ which shows that the equation \eqref{eq:GammaBase3All} holds for all $\eta \geq \eta^*$.
\end{proof}

The equation \eqref{eq:GammaBase3All} can be solved using the same method as the one we applied for the binary base in section \ref{sec:BinaryBase}. Before solving the equation \eqref{eq:GammaBase3All}, let's see an example for $\gamma_3(3)$ and for $\gamma_3(4)$ in order to get an intuition on the general solution. Knowing that $\gamma_3(2) = 13$, we have
\begin{align} \label{eq:Gamma3_3_Solved}
	\gamma_3(3) &= 2 \cdot 3^{\frac{\gamma_3(2) - 1}{4}} - 1 \notag \\
				&= 2 \cdot 3^{\frac{13 - 1}{4}} - 1 \notag \\
				&= 2 \cdot 3^{\frac{(14 - 1) - 1}{4}} - 1 \notag \\
				&= 2 \cdot 3^{\frac{-1}{2}} \cdot 3^{\frac{7}{2}} - 1.
\end{align}
Then, for $\gamma_3(4)$ we have
\begin{align} \label{eq:Gamma3_4_Solved}
	\gamma_3(4) &= 2 \cdot 3^{\frac{\gamma_3(3) - 1}{4}} - 1 \notag \\
				&= 2 \cdot 3^{\frac{(2 \cdot 3^{\frac{-1}{2}} \cdot 3^{\frac{7}{2}} - 1) - 1}{4}} - 1 \notag \\
				&= (2 \cdot 3^{\frac{-1}{2}}) \cdot 3^{(\frac{1}{2} \cdot 3^{\frac{-1}{2}}) \cdot 3^{\frac{7}{2}}} - 1.
\end{align}
Continuing with $\gamma_3(5)$ we would see the pattern $\frac{1}{2} \cdot 3^{\frac{-1}{2}} \cdot 3$ repeating. However, in order to fit with this pattern, the coefficient $2$ in \eqref{eq:Gamma3_3_Solved} and \eqref{eq:Gamma3_4_Solved} has to be written $4 \cdot \frac{1}{2}$ instead.

\begin{theorem} \label{thm:GammaBase3All}
For all $\eta > \eta^*$, we have 
\begin{equation}\label{eq:GammaBase3Height}
\gamma_3(\eta) = 4\left(\frac{3^{\frac{-1}{2}}}{2} 3 \uparrow\uparrow (\eta - \eta^*)\right)^{\frac{7}{2}} - 1.
\end{equation}
\end{theorem}

\begin{proof}
We proceed by induction on $\eta$. Since $\eta^* = 2$ in ternary base, we have by definition \ref{dfn:KnuthAdapted} that $4\left(\frac{3^{\frac{-1}{2}}}{2} 3 \uparrow\uparrow 1\right)^{\frac{7}{2}} - 1 = 4\left(\frac{3^{\frac{-1}{2}}}{2} 3^{\frac{7}{2}} \right) - 1 = 2 \cdot 3^3 - 1 = \gamma_3(3)$. Suppose that the equation \eqref{eq:GammaBase3Height} is verified for a certain $\eta > \eta*$. In virtue of the lemma \ref{lem:GammaBase3General}, we have
\begin{equation} \label{eq:Gamma3_Proof_First_Step}
	\gamma_3(\eta + 1) = 2 \cdot 3^{\frac{\gamma_3(\eta) - 1}{4}} - 1.
\end{equation}
We apply the induction hypothesis on the equation \eqref{eq:Gamma3_Proof_First_Step} and obtain
\begin{align}
	\gamma_3(\eta + 1) &= 2 \cdot 3^{\frac{4\left(\frac{3^{\frac{-1}{2}}}{2} 3 \uparrow\uparrow (\eta - \eta^*)\right)^{\frac{7}{2}} - 1 - 1}{4}} - 1 \notag \\ 
	 				   &= 2 \cdot 3^{\left(\frac{3^{\frac{-1}{2}}}{2} 3 \uparrow\uparrow (\eta - \eta^*)\right)^{\frac{7}{2}} - \frac{1}{2}} - 1 \notag \\ 
	                   &= 4 \frac{3^{\frac{-1}{2}}}{2}3^{\left(\frac{3^{\frac{-1}{2}}}{2} 3 \uparrow\uparrow (\eta - \eta^*)\right)^{\frac{7}{2}}} - 1. 
\end{align}
In virtue of the definition \ref{dfn:KnuthAdapted}, it follows that 
\begin{equation}
	\gamma_3(\eta + 1) = 4\left(\frac{3^{\frac{-1}{2}}}{2} 3 \uparrow\uparrow (\eta - \eta^* + 1)\right)^{\frac{7}{2}} - 1.
\end{equation}
Therefore, the equation \eqref{eq:GammaBase3Height} is verified for all $\eta > \eta^*$.
\end{proof}

We conclude this section with the following questions:
\begin{enumerate}
	\item If this method can be generalized, what would be the general equation solving the equation \eqref{eq:General_Gamma_Height} for all $B \geq 2$ and $\eta_B > \eta_B^*$?
	\item If the method cannot be generalized, what method would solve the equation \eqref{eq:General_Gamma_Height} for all $B \geq 2$ and $\eta_B > \eta_B^*$? 
\end{enumerate}

\section*{Acknowledgments}
I want to thank Maxime Toussaint, candidate to Ph.D in Computer Science at University of Sherbrooke, for his time to review this paper and his valuable advices.


\begin{thebibliography}{0}
\bibitem{Grundman01}
H.G. Grundman and E.A. Teeple, 
{Generalized happy numbers} 
Fibonacci Quart. {\bf 39}(2001), 462--466.

\bibitem{Grundman03}
H.G. Grundman and E.A. Teeple, 
{Heights of happy numbers and cubic happy numbers} 
Fibonacci Quart. {\bf 41}(2003), 301--306.

\bibitem{Knuth76}
Knuth, D. E. 
{Mathematics and Computer Science: Coping with Finiteness. Advances in Our Ability to Compute are Bringing Us Substantially Closer to Ultimate Limitations.} 
Science {\bf 194}(1976), 1235--1242.

\bibitem{Mei16}
May Mei, Andrew Read-McFarland 
{Numbers and the Heights of their Happiness} 
arXiv:1511.01441v2 (2016), 1--5.

\bibitem{Sloane64} 
N. J. A. Sloane. The On-Line Encyclopedia of Integer Sequences, http://oeis.org. Sequence A001348.

\bibitem{Tianxin08}
Tianxin Cai and Xia Zhou
{On the Heights of Happy Numbers} 
Rocky Mountain J. Math. {\bf 38}(2008), no.~6, 1921--1926.
\end{thebibliography}
\end{document}